\documentclass[12pt,a4paper]{article}
\usepackage[cp1250]{inputenc}
\usepackage[english]{babel}
\usepackage{amsmath,amssymb,amsthm,amsfonts}
\usepackage{tikz}
\usepackage{mathtools}
\usepackage{bussproofs}

\author{Daniyar Shamkanov\thanks{Supported by the Russian Foundation for Basic Research, and Russian Presidential Council for Support of Leading Scientific Schools.}\\ \normalsize{\textit{Steklov Mathematical Institute of the Russian Academy of Sciences}}\\
\normalsize{\textit{National
Research University Higher School of Economics}}\\
\normalsize{\textit{daniyar.shamkanov@gmail.com}}\\
}
\title{Circular Proofs for G\"{o}del-L\"{o}b Logic}

\newtheorem{thm}{Theorem}[section]
\newtheorem{prop}[thm]{Proposition}
\newtheorem{lem}[thm]{Lemma}

\theoremstyle{remark}

\EnableBpAbbreviations
\newcommand{\tikzmark}[1]{\tikz[overlay,remember picture] \node (#1) {};}

\begin{document}
\maketitle

\begin{abstract}
We present a sequent-style proof system for provability logic $\mathsf{GL}$ that admits so-called circular proofs. For these proofs, the graph underlying a proof is not a finite tree but is allowed to contain cycles. As an application, we establish Lindon interpolation for $\mathsf{GL}$ syntactically.
\end{abstract}

\section{Introduction}
G\"{o}del-L\"{o}b logic $\mathsf{GL}$ is a modal logic, which describes all universally valid principals of the formal provability in Peano arithmetic.
This logic has attracted strong interest and a lot of effort has been directed to the search of an adequate, cut-free sequent calculus for $\mathsf{GL}$ (see \cite{Neg05, Pog09, GorRam12}). In the present paper we propose a sequent-style proof system for $\mathsf{GL}$ that admits so-called circular proofs. A circular proof can be formed from an ordinary derivation by identifying each non-axiomatic leaf of the derivation tree with an identical interior node via so-called "back-links" (see Section 3). These kind of proofs appears to be an interesting alternative to traditional proofs for logics that contain fixed-point operators (cf. \cite{BroGorPet12}).
Since $\mathsf{GL}$ can be seen as a fragment of the modal $\mu$-calculus (see \cite{Ben06, Vis05, AlFa09}), there should be applications of circular proofs to $\mathsf{GL}$. We present one such application in the present paper.

Though G\"{o}del-L\"{o}b logic possesses Lyndon interpolation \cite{Sha11}, there were seemingly no syntactic proofs for this result.  
It is unclear how Lyndon interpolation can be obtained from previously introduced sequent systems \cite{Neg05, Pog09, GorRam12} by direct proof-theoretic arguments because these systems contain inference rules where the principal formula changes polarity from the conclusion to the premise. In this note, we obtain Lyndon interpolation for $\mathsf{GL}$ by applying circular proofs.
  
In the next section, we recall a variant of the standard sequent calculus for $\mathsf{GL}$. Then we introduce the circular proof system and prove its equivalence to the standard one. In the final section, we establish Lyndon interpolation for $\mathsf{GL}$ syntactically.
\section{Preliminaries}

\textit{Formulas} of $\mathsf{GL}$, denoted by $A$, $B$, $C$, are built up as follows:
$$ A ::= P \,\,|\,\, \overline{P} \,\,|\,\, \top \,\,|\,\, \bot \,\,|\,\, (A \wedge A) \,\,|\,\,(A \vee A) \,\,|\,\, \Box A \,\,|\,\, \Diamond A \;, $$
where $P$ and $\overline{P}$ stand for atoms and their complements. 

The \textit{negation} $\overline{A}$ of a formula $A$ is defined in the usual way by De Morgan's laws, the law of double negation and the duality laws for modal operators, i.e. we inductively define
\begin{gather*}
\overline{(P)} := \overline{P},\qquad
\overline{\overline{P}} := P,\\
\overline{\top} := \bot,\qquad
\overline{\bot} := \top,\\
\overline{(A \wedge B)} := (\overline{A} \vee \overline{B}),\qquad
\overline{(A \vee B)} := (\overline{A} \wedge \overline{B}),\\
\overline{\Box A} := \Diamond \overline{A},\qquad
\overline{\Diamond A} := \Box \overline{A}.
\end{gather*}

We also put
\begin{gather*}
A\rightarrow B := \overline{A} \vee B, \qquad A \leftrightarrow B := (A\rightarrow B) \wedge
(B\rightarrow A).
\end{gather*}

The Hilbert-style axiomatization of $\mathsf{GL}$ is as follows:

\textit{Axioms:}
\begin{itemize}
\item[(i)] Boolean tautologies;
\item[(ii)] $\Box (A \rightarrow B) \rightarrow (\Box A \rightarrow \Box B)$;
\item[(iii)] $\Box ( \Box A \rightarrow A) \rightarrow \Box A$.
\end{itemize}

\textit{Rules:} modus ponens, $A / \Box A$. \\

Now we remind the reader a variant of the standard sequent-style formulation of $\mathsf{GL}$. A \textit{sequent} is a finite multiset of formulas denoted by $\Gamma$, $\Delta$, $\Sigma$. Sequents are often written without any curly braces, and the comma in the expression $\Gamma, \Delta$ means the multiset union. 

For a sequent $\Gamma = A_1, \ldots, A_n$, we put
\begin{gather*}
\Diamond \Gamma := \Diamond A_1, \ldots, \Diamond A_n 
\quad \text{and} \quad
\Gamma^\sharp :=
\begin{cases}
\bot & \text{if $n =0$,}\\
 A_1 \vee \ldots \vee A_n & \text{otherwise.}
\end{cases}
\end{gather*}

The system $\mathsf{GL_{Seq}}$, which is a sequent-style formulation of $\mathsf{GL}$, is defined by the following initial sequents and inference rules: 

\begin{gather*}
\AXC{ $\Gamma, A, \overline{A} $}
\DisplayProof \qquad
\AXC{ $\Gamma , \top $}
\DisplayProof
\end{gather*}
\begin{gather*}
\AXC{$\Gamma , A $}
\AXC{$\Gamma , B $}
\LeftLabel{$\mathsf{\wedge}$}
\BIC{$\Gamma , A \wedge B$}
\DisplayProof \qquad
\AXC{$\Gamma , A,B $}
\LeftLabel{$\mathsf{\vee}$}
\UIC{$\Gamma , A \vee B$}
\DisplayProof \\\\
\AXC{$\Gamma , \Diamond \Gamma , \Diamond \overline{A} ,A$}
\LeftLabel{$\mathsf{\Box_{GL}}$}
\UIC{$\Diamond \Gamma , \Box A ,\Delta $}
\DisplayProof \;.
\end{gather*}
The standard sequent calculus $\mathsf{GLS}$ for $\mathsf{GL}$ introduced by G.~Sambin and S.~Valentini (\cite{SamVal80, SamVal82}) was based on two-sided sequents $\Gamma \Rightarrow \Delta$ where $\Gamma$ and $\Delta$ were sets. To facilitate proof-theoretic treatment of $\mathsf{GL}$, we deal with one-sided sequents built from multisets, and consider formulas in negation normal form only. 

A syntactic proof of the cut-elimination theorem for $\mathsf{GLS}$ was given by S.~Valentini in \cite{Val83}. For a multiset-based version of $\mathsf{GLS}$, the analogous result was proved by R.~Gor\'{e} and R.~Ramanayake in \cite{GorRam12}. From the result of \cite{GorRam12} and admissibility of structural rules for $\mathsf{GL_{Seq}}$, we have the cut-elimination theorem for $\mathsf{GL_{Seq}}$ immediately. Hence we have the following propositions:
\begin{prop} \label{inf-cut}
The L\"{o}b rule and the cut rule
\[\AXC{$\Gamma , \Diamond \Gamma , \Diamond \overline{A} ,A$} \LeftLabel{$\mathsf{L\ddot{o}b}$} \UIC{$\Gamma , \Diamond \Gamma ,A$} \DisplayProof \qquad \AXC{$\Gamma, A$} \AXC{$\Gamma, \overline{A}$} \LeftLabel{$\mathsf{cut}$} \BinaryInfC{$\Gamma$} \DisplayProof\]
are admissible for $\mathsf{GL_{Seq}}$. 
\end{prop}

\begin{prop} \label{prop}
$\mathsf{GL_{Seq}} \vdash \Gamma \Longleftrightarrow \mathsf{GL} \vdash \Gamma^\sharp $. 
\end{prop}

\section{Non-well-founded and Circular Proofs}
A \textit{derivation} in a sequent calculus is a finite tree whose nodes are marked by sequents that is constructed according to the rules of the sequent calculus. 
In the present section we also deal with proof systems that admit non-well-founded derivation trees, which we call \textit{$\infty$-derivations}. A ($\infty$-)\textit{proof} is defined as a ($\infty$-)derivation, where all leaves are labeled with initial sequents.  

A tree is called \textit{regular} if it contains only finitely many distinct subtrees. Equivalently, a regular tree is a tree that can be obtained by unraveling from a finite directed graph. A \textit{regular $\infty$-derivation} is an $\infty$-derivation with regular derivation tree.

A \textit{circular derivation} is a pair $(\kappa, d)$, where $\kappa$ is an ordinary derivation and $d$ is a \textit{back-link function} assigning to some leaf $x$ an interior node $y$ with an identical sequent such that $y$ lies on the path from the root to the leaf. A \textit{circular proof} is a circular derivation such that every leaf is either connected by the back-link function, or is marked by an initial sequent.

Recall that a sequent calculus for the standard modal logic $\mathsf{K4}$, denoted by $\mathsf{K4_{Seq}}$, is obtained from $\mathsf{GL_{Seq}}$ by replacing the rule $\mathsf{\Box_{GL}}$ by the modal rule:
\[\AXC{$\Gamma , \Diamond \Gamma , A$}
\LeftLabel{$\mathsf{\Box}$}
\UIC{$\Diamond \Gamma , \Box A ,\Delta $}
\DisplayProof \;.
\]
We introduce the systems $\mathsf{GL_{circ}}$ and $\mathsf{GL_{\infty}}$ obtained from $\mathsf{K4_{Seq}}$ by admitting circular proofs and arbitrary $\infty$-proofs, respectively. For example, consider a circular proof for L\"{o}b's axiom $\Box (\Box P \rightarrow P) \rightarrow \Box P$:

\begin{gather*}
   \AXC{$\tikzmark{A}\Box P \wedge \overline{P} , \Diamond (\Box P \wedge \overline{P}) , P
$}
\LeftLabel{$\mathsf{\Box}$} 
\UIC{$\Box P , \Diamond (\Box P \wedge \overline{P}) , P$} 
\AXC{$\overline{P} , \Diamond (\Box P \wedge \overline{P}) , P$}
\LeftLabel{$\mathsf{\wedge}$}
\BIC{$\tikzmark{b} \Box P \wedge \overline{P} , \Diamond (\Box P \wedge \overline{P}) , P $}
\LeftLabel{$\mathsf{\Box}$} 
\UIC{$\Diamond (\Box P \wedge \overline{P}), \Box P$}
\LeftLabel{$\mathsf{\vee}$}
\UIC{$\Diamond (\Box P \wedge \overline{P}) \vee \Box P$}
 \DisplayProof
 \begin{tikzpicture}[overlay,remember picture, >=latex, distance=2.7cm]
    \draw[->, thick] (A.north west) to [out=180,in=245] (b.west);
 \end{tikzpicture}
\end{gather*}
\begin{center}
\textbf{Fig. 1}
\end{center}
We stress that the notion of circular derivation is essentially coincide with the notion of regular $\infty$-derivation. Obviously, every circular derivation can be unraveled to the regular $\infty$-derivation. The converse also holds:
\begin{prop}[Chapter 6 in \cite{Bro06}] \label{circ}
Any regular $\infty$-derivation can be seen as the unraveling of a circular derivation.
\end{prop}

The rest of the section is devoted to proving the following result.
\begin{thm} \label{thm}
$\mathsf{GL_{Seq}} \vdash \Gamma \Longleftrightarrow \mathsf{GL_\mathsf{\infty}} \vdash \Gamma \Longleftrightarrow \mathsf{GL_{circ}} \vdash \Gamma$. 
\end{thm}
\begin{lem}
$\mathsf{GL_{Seq}} \vdash \Gamma \Longrightarrow \mathsf{GL_{\infty}} \vdash \Gamma$. 
\end{lem}
\begin{proof}
Assume $\pi$ is a proof of $\Gamma$ in $\mathsf{GL_{Seq}}$. By corecursion, we define the $\infty$-proof $h(\pi)$ of $\Gamma$ in $\mathsf{GL_{\infty}}$. The function $h$ maps initial sequents into initial sequents and commutes with logical rules.
The case of the modal rule is as follows:
\begin{gather*}
\AXC{$\tau$}
\noLine
\UIC{\vdots}
\noLine
\UIC{$\Delta , \Diamond \Delta , \Diamond \overline{A} ,A$}
\LeftLabel{$\mathsf{\Box_{GL}}$}
\UIC{$\Diamond \Delta , \Box A ,\Sigma $}
\DisplayProof 
\xmapsto{\quad  h \quad}
\AXC{$h(f(\tau))$}
\noLine
\UIC{\vdots}
\noLine
\UIC{$\Delta , \Diamond \Delta  ,A$}
\LeftLabel{$\mathsf{\Box}$}
\RightLabel{ ,}
\UIC{$\Diamond \Delta , \Box A ,\Sigma $}
\DisplayProof 
\end{gather*}
where $f(\tau)$ is a proof of $ \Delta , \Diamond \Delta ,A $ in $\mathsf{GL_{Seq}}$ existed by admissibility of the L\"{o}b rule (Proposition \ref{inf-cut}). 
Then $h(\pi)$ proves $\Gamma$ in $\mathsf{GL_{\infty}}$. 
\end{proof}

We say that an inference rule is admissible for $\mathsf{GL_{\infty}}$ if, for every instance of the rule, the
conclusion is provable whenever all premises are provable.
\begin{lem} \label{inf} The inverses of logical rules of $\mathsf{GL_{\infty}}$
\[\AXC{$\Gamma , A \wedge B $}
\UIC{$\Gamma , A$}
\DisplayProof \qquad
\AXC{$\Gamma , A \wedge B $}
\UIC{$\Gamma , B$}
\DisplayProof \qquad
\AXC{$\Gamma , A \vee B $}
\UIC{$\Gamma , A, B$}
\DisplayProof \qquad
\AXC{$\Gamma , \bot $}
\UIC{$\Gamma $}
\DisplayProof
\] are admissible for $\mathsf{GL_{\infty}}$.
\end{lem}

\begin{proof}
Simple transformations of proofs.
\end{proof}

\begin{lem} The rules of weakening and contraction
$$\AXC{$\Gamma$} \LeftLabel{$\mathsf{wk}$} \UIC{$\Gamma, A$} \DisplayProof \qquad \AXC{$\Gamma, A, A$} \LeftLabel{$\mathsf{ctr}$} \UIC{$\Gamma, A$} \DisplayProof $$
are admissible for $\mathsf{GL_{\infty}}$.
\end{lem}
\begin{proof}
Admissibility of the weakening rule follows from the built-in weakening within initial sequents and the rule $\Box$. 

We prove the admissibility of contraction by induction on the structure of $A$.

The cases of $P$, $\overline{P}$, $\top$, $\bot$, $\wedge$, $\vee$ and $\Box B$ are established standardly using the induction hypothesis, Lemma \ref{inf}, and the definition of the rule $\Box$.

The case $A = \Diamond B$. Assume $\pi$ is an $\infty$-proof of $ \Gamma \cup \{ \Diamond B , \Diamond B \}$ in $\mathsf{GL_{\infty}}$. By corecursion, we define the $\infty$-proof $h(\pi)$ of $ \Gamma \cup \{ \Diamond B \}$ in $\mathsf{GL_{\infty}}$. The function $h$ maps initial sequents into initial sequents and commutes with logical rules. Consider the case of the modal rule:
\begin{gather*}
\AXC{$\tau$}
\noLine
\UIC{\vdots}
\noLine
\UIC{$\Delta , \Diamond \Delta , C$}
\LeftLabel{$\mathsf{\Box}$}
\RightLabel{ .}
\UIC{$\Diamond \Delta , \Box C ,\Sigma $}
\DisplayProof 
\end{gather*}
If $\Diamond B \in \Sigma$, then $h$ simply erases a copy of $\Diamond B$ from the conclusion of the proof.
Otherwise, $\Delta =  \Delta^\prime \cup \{ B , B \}$ and $h$ maps this $\infty$-proof to the following:
\begin{gather*}
\AXC{$\tau$}
\noLine
\UIC{\vdots}
\noLine
\UIC{$ B, B, \Delta^\prime , \Diamond B , \Diamond B ,\Diamond \Delta^\prime  ,C$}
\LeftLabel{$\mathsf{\Box}$}
\UIC{$\Diamond B  ,\Diamond B  ,\Diamond \Delta^\prime , \Box C ,\Sigma$}
\DisplayProof 
\xmapsto{\quad  h \quad}
\AXC{$h(f(\tau))$}
\noLine
\UIC{\vdots}
\noLine
\UIC{$ B, \Delta^\prime , \Diamond B , \Diamond \Delta^\prime  ,C$}
\LeftLabel{$\mathsf{\Box}$}
\RightLabel{ ,}
\UIC{$\Diamond B  ,\Diamond \Delta^\prime , \Box C ,\Sigma  $}
\DisplayProof 
\end{gather*}
where $f(\tau)$ is an $\infty$-proof obtained from $\tau$ by contracting a copy of $B$ from the conclusion of $\tau$. This transformation is admissible by the inductive hypothesis. 
Then $h(\pi)$ proves $\Gamma \cup \{ \Diamond B\}$ in $\mathsf{GL_{\infty}}$.

\end{proof}
\begin{lem}
$\mathsf{GL_{\infty}} \vdash \Gamma \Longrightarrow \mathsf{GL_{circ}} \vdash \Gamma$. 
\end{lem}
\begin{proof}
For a sequent $\Delta$, we denote its underlying set by $\Delta^S$. We have $\Delta = \Delta^S , \Delta^\prime$. 

Let $\pi$ be an $\infty$-proof of $\Gamma$ in $\mathsf{GL_{\infty}}$. By corecursion, we define the $\infty$-proof $h(\pi)$ of $\Gamma$, which contains only finitely many different sequents. The function $h$ maps initial sequents into initial sequents and commutes with logical rules.
The case of the modal rule is as follows:
\begin{gather*}
\AXC{$\tau$}
\noLine
\UIC{\vdots}
\noLine
\UIC{$\Delta, \Diamond \Delta , A $}
\LeftLabel{$\mathsf{\Box}$}
\UIC{$\Diamond \Delta , \Box A ,\Sigma $}
\DisplayProof 
\xmapsto{\quad  h \quad}
\AXC{$h(f(\tau))$}
\noLine
\UIC{\vdots}
\noLine
\UIC{$\Delta^S , \Diamond  \Delta^S  ,A$}
\LeftLabel{$\mathsf{\Box}$}
\RightLabel{ ,}
\UIC{$\Diamond \Delta^S , \Box A  , \Sigma , \Diamond \Delta^\prime $}
\DisplayProof 
\end{gather*}
where $f(\tau)$ is an $\infty$-proof of $\Gamma^S , \Diamond  \Gamma^S  ,A$ in $\mathsf{GL_{\infty}}$. This proof exists by the admissibility of the contraction rule. 

Note that all formulas from $h(\pi)$ are subformulas of the formulas from $\Gamma$. Since there are only finitely many different premises of the rule $\Box$ in $h(\pi)$, and, for the other rules, conclusions are longer than premises, we have that $h(\pi)$ contains only finitely many different sequents. From any $\infty$-proof with finitely many different sequents, a regular $\infty$-proof can be obtained immediately. Hence, we have a circular proof of $\Gamma$ by Proposition \ref{circ}.

\end{proof}

For a circular derivation  $\pi = (\kappa, d)$, the set of assumption leafs of $\pi$ is the set of non-axiomatic leafs of $\kappa$ that are not connected by the back-link function $d$. We say that an assumption leaf is boxed if there is an application of the rule $\Box$ on the path from the root to the leaf. Let $BH(\pi)$ be the set of boxed assumption leafs of $\pi$, and $H(\pi)$ be the set of all other assumption leafs. For a subderivation $\tau$ of $\kappa$, we define the back-link function $d_\tau$ as the the set of all links from $d$ with images inside $\tau$.

Recall that $\boxdot A$ is an abbreviation for $A \wedge \Box A$.
\begin{lem}
$\mathsf{GL_{circ}} \vdash \Gamma \Longrightarrow \mathsf{GL_{Seq}} \vdash \Gamma$. 
\end{lem}
\begin{proof}
Given a circular derivation $\pi = (\kappa, d)$ of $\Gamma$, we claim that $$\mathsf{GL} \vdash \bigwedge \lbrace \boxdot \Delta_a^\sharp \colon a \in H(\pi) \rbrace \wedge \bigwedge \lbrace \Box \Delta_a^\sharp \colon a \in BH(\pi) \rbrace \rightarrow \Gamma^\sharp \;,$$ where $\Delta_a$ is the sequent of $a$. The proof is by induction on the structure of $\kappa$. 

If $\kappa$ consists of a single sequent, then our claim is obvious. Otherwise, consider the last application of an inference rule in $\kappa$. We have three cases:
\begin{gather*}
\AXC{$\tau_1$}
\noLine
\UIC{\vdots}
\noLine
\UIC{$\Delta, A $}
\AXC{$\tau_2$}
\noLine
\UIC{\vdots}
\noLine
\UIC{$\Delta, B $}
\LeftLabel{$\wedge$}
\BIC{$\Delta, A \wedge B$}
\DisplayProof \qquad 
\AXC{$\tau$}
\noLine
\UIC{\vdots}
\noLine
\UIC{$\Delta, A, B $}
\LeftLabel{$\vee$}
\UIC{$\Delta, A \vee B$}
\DisplayProof \qquad 
\AXC{$\tau$}
\noLine
\UIC{\vdots}
\noLine
\UIC{$\Delta, \Diamond \Delta, A $}
\LeftLabel{$\Box$}
\UIC{$\Diamond \Delta, \Box A,  \Sigma$}
\DisplayProof 
\end{gather*}
Suppose the conclusion of $\kappa$ is not connected by the back-link function. Recall that
\begin{gather*}
\mathsf{GL} \vdash (\Delta, A)^\sharp \wedge (\Delta, B)^\sharp \rightarrow (\Delta, A \wedge B)^\sharp, \qquad \mathsf{GL} \vdash (\Delta, A, B)^\sharp \rightarrow (\Delta, A \vee B)^\sharp , \\ \mathsf{GL} \vdash \Box (\Delta, \Diamond \Delta, A)^\sharp \rightarrow (\Diamond \Delta, \Box A,  \Sigma)^\sharp .
\end{gather*}
We see that the claim follows from the induction hypothesis for $(\tau, d_\tau)$ (or for $(\tau_1, d_{\tau_1})$ and $(\tau_2, d_{\tau_2})$) immediately.

The only remaining case is that the conclusion of $\kappa$ is connected by the back-link function. Recall that there is always an application of the rule $\Box$ on the path between two nodes connected by the back-link function. By the induction hypothesis and from the previous cases, we have:
$$\mathsf{GL} \vdash \bigwedge \lbrace \boxdot \Delta_a^\sharp \colon a \in H(\pi) \rbrace \wedge \bigwedge \lbrace \Box \Delta_a^\sharp \colon a \in BH(\pi) \rbrace \wedge \Box \Gamma^\sharp \rightarrow \Gamma^\sharp \;.$$ 
Since the L\"{o}b rule is admissible for $\mathsf{GL}$, the assumption $\Box \Gamma^\sharp$ can be dropped out. Hence,  
$$\mathsf{GL} \vdash \bigwedge \lbrace \boxdot \Delta_a^\sharp \colon a \in H(\pi) \rbrace \wedge \bigwedge \lbrace \Box \Delta_a^\sharp \colon a \in BH(\pi) \rbrace \rightarrow \Gamma^\sharp \;.$$
Now if $\pi$ is a circular proof of $\Gamma$, then $\mathsf{GL} \vdash \Gamma^\sharp$. 
By Proposition \ref{prop}, $\mathsf{GL_{Seq}} \vdash \Gamma$. 
\end{proof}

\section{Lyndon Interpolation Syntactically}

Lyndon interpolation for G\"{o}del-L\"{o}b logic was established in \cite{Sha11} on the basis of Kripke semantics. Here we present a proof-theoretic argument for the same result. In fact, we establish a simple strengthening of this property.

Atoms $P$ and their complements $\overline{P}$ are called \textit{literals}. 
By $u (A)$, denote the set of literals $L$ occurring in $A$ out of the scope of all modal operators. For atoms $P$ and their complements $\overline{P}$, we also consider new symbols of the form $P^\circ$ and $\overline{P}^\circ$ and call them \textit{marked literals}. By $v (A)$, denote the set of marked literals $ L^\circ$ such that there is an occurrence of $L$ in $A$ within the scope of a modal operator. Put $w(A) = u(A) \cup v(A)$.

\begin{thm}[Lyndon interpolation] \label{Lyn}
If $\mathsf{GL} \vdash A \rightarrow B$, then there is a formula $C$, called an interpolant of $A \rightarrow B$, such that $w (C) \subset w (A) \cap w (B)$, and \[\mathsf{GL }\vdash A \rightarrow C , \qquad \mathsf{GL} \vdash C \rightarrow B.\] 
\end{thm} 
For a sequent $\Gamma_1, \Gamma_2$, the expression of the form $\Gamma_1 \: | \;\Gamma_2$ is called its \textit{splitting}. An \textit{interpolant of a split sequent $\Gamma_1\: | \: \Gamma_2$} is defined as an
interpolant of the formula $\overline{(\Gamma^\sharp_1)} \rightarrow \Gamma^\sharp_2$. Thus, it is sufficient to find interpolants for all splittings of provable sequents. 

The standard proof-theoretic strategy of constructing an interpolant is based on the following observations:
\begin{enumerate}
\item given an application of an inference rule of a cut-free sequent calculus, every splitting of the conclusion produces splittings of the premises preserving ancestor relationship;
\item  there is an explicit definition of the interpolant for the split sequent in the conclusion from interpolants of the split sequents in the premises.
\end{enumerate}
For the rules of $\mathsf{GL_{circ}}$, these observations can be summed up in the following split inference system, where in the parentheses we present the process of constructing an interpolant:
\begin{gather*}
\AXC{ $(\top) \;\; \Gamma_1 \; | \; \top, \Gamma_2$}
\DisplayProof \qquad
\AXC{ $(\bot) \;\; \Gamma_1 , \top  \; | \; \Gamma_2$}
\DisplayProof\\\\
\AXC{$(\bot) \;\; \Gamma_1, A, \overline{A} \; | \; \Gamma_2$}
\DisplayProof \qquad
\AXC{$(\overline{A}) \;\;\Gamma_1, A \; | \;\overline{A} , \Gamma_2$}
\DisplayProof \qquad
\AXC{$(\top) \;\;\Gamma_1 \; | \; A, \overline{A}, \Gamma_2 $}
\DisplayProof \\\\
\AXC{$(C) \;\;\Gamma_1 , A \; | \; \Gamma_2  $}
\AXC{$(D) \;\; \Gamma_1 , B \; | \; \Gamma_2 $}
\LeftLabel{$\mathsf{\wedge}_l$}
\BIC{$({C \vee D}) \;\; \Gamma_1 , A \wedge B \; | \; \Gamma_2$}
\DisplayProof  \qquad
\AXC{$(C) \;\; \Gamma_1 , A,B \; | \; \Gamma_2 $}
\LeftLabel{$\mathsf{\vee}_l$}
\UIC{$(C) \;\; \Gamma_1 , A \vee B \; | \; \Gamma_2$}
\DisplayProof 
\end{gather*}
\begin{gather*}
\AXC{$(C) \;\; \Gamma_1 \; | \; \Gamma_2 , A $}
\AXC{$(D) \;\;\Gamma_1 \; | \; \Gamma_2 , B $}
\LeftLabel{$\mathsf{\wedge}_r$}
\BIC{$(C\wedge D) \;\;\Gamma_1 \; | \; \Gamma_2 , A \wedge B$}
\DisplayProof \qquad
\AXC{$(C) \;\; \Gamma_1 \; | \; \Gamma_2  , A,B$}
\LeftLabel{$\mathsf{\vee}_r$}
\UIC{$(C) \;\; \Gamma_1 \; | \; \Gamma_2 , A \vee B$}
\DisplayProof
\end{gather*}
\begin{gather*}
\AXC{$(C) \;\; \Gamma_1 , \Diamond \Gamma_1  ,A \; | \; \Gamma_2 , \Diamond \Gamma_2$}
\LeftLabel{$\mathsf{\Box}_l$}
\UIC{$(\Diamond C) \;\; \Diamond \Gamma_1 , \Box A ,\Delta_1 \; | \; \Diamond \Gamma_2 , \Delta_2 $}
\DisplayProof \quad
\AXC{$(C) \;\; \Gamma_1 , \Diamond \Gamma_1   \; | \; \Gamma_2 , \Diamond \Gamma_2 ,A$}
\LeftLabel{$\mathsf{\Box}_r$}
\UIC{$ (\Box C) \;\; \Diamond \Gamma_1  ,\Delta_1 \; | \; \Diamond \Gamma_2 ,  \Box A ,\Delta_2 $}
\DisplayProof 
\end{gather*}
\begin{center}
\textbf{Fig. 2}
\end{center}
Now the notion of a \textit{split circular derivation} can be defined naturally with the proviso that nodes connected by the back-link function marked by identical split sequents.

\begin{lem}\label{splitproof}
If there is a circular proof of $ \Gamma_1 , \Gamma_2$ in $\mathsf{GL_{circ}}$, then there is a split circular proof of $ \Gamma_1 \: | \: \Gamma_2 $. 

\end{lem}

\begin{proof}
Assume $\pi$ is a circular proof of $\Gamma_1 , \Gamma_2$. By bottom-up splitting of sequents from $\pi$, we obtain a split $\infty$-proof of $\Gamma_1 \: | \: \Gamma_2$. Since there are only finitely many different sequents in $\pi$, there are only finitely many different split sequets in the split $\infty$-proof of $\Gamma_1 \: | \: \Gamma_2$. From any $\infty$-proof with finitely many different sequents, a regular $\infty$-proof can be obtained immediately. Hence, we have a split circular proof of $\Gamma_1 \: | \: \Gamma_2$ by Proposition \ref{circ}.

\end{proof}

Now we state a variant of the salient fixed-point theorem for $\mathsf{GL}$. 
For a set $S$ of literals $L$ and marked literals $L^\circ$, define $S^\circ = \{ L^\circ \colon \text{$L \in S$ or $L^\circ \in S$} \}$. Put $w^* (A) = w(A) \cup w^\circ (A)$. 
\begin{thm} \label{Fixed-point}
Let $A$ be a formula in which $P$ and $\overline{P}$ only occur within the scope of modal operators. Then there is a formula $F$ such that $w (F) \subset (w^* (A)\cup w^*(\overline{A})) \setminus \{ P^\circ, \overline{P^\circ}\}$,  and
\[\mathsf{GL} \vdash \boxdot (P \leftrightarrow A) \leftrightarrow \boxdot (P \leftrightarrow F) \: .\]
Moreover, if $A$ does not contain $\overline{P}$, then $w (F) \subset w^* (A) \setminus \{ P^\circ\}$.
\end{thm}
Note that this variant of the theorem can be obtained by a syntactic procedure without applying any kind of interpolation (see \cite{Lin96, Smo04}).

For a sequent $\Gamma$, define $\overline{\Gamma}$ as the multiset of negations of formulas from $\Gamma$.  Let $w (\Gamma)$ be the set $ \bigcup \{ w (A) \colon A \in \Gamma \} $. 
\begin{proof}[Proof of Theorem \ref{Lyn}]
It is sufficient to find interpolants for all splittings of provable sequents.
Assume we have a circular proof of  $ \Gamma_1 , \Gamma_2$. Applying Lemma \ref{splitproof}, we find a split circular proof of $ \Gamma_1 \; | \; \Gamma_2 $, and construct an interpolant for this split sequent as follows.

Consider any split circular derivation $\pi=(\kappa, d)$ of $ \Gamma_1 \: | \: \Gamma_2 $.
For any non-axiomatic leaf $a$ of $\kappa$, we fix two unknowns $X_a$ and $w_a$. The unknown atom $X_a$ stands for an interpolant of the split sequent in the leaf $a$. The intended interpretation of $w_a$ is $w(X_a)$. Let the split sequent of the leaf $a$ be $\Delta_1 \: | \: \Delta_2$. Define $I_a$ as the formula $(\overline{(\Delta^\sharp_1)} \rightarrow X_a) \wedge (X_a \rightarrow \Delta^\sharp_2)$ and $I_a^\prime$ as the statement $w_a \subset (w(\overline{\Delta}_1) \cap w( \Delta_2))$. For a formula $D$, which can contain new atoms $X_a$, let $w_X(D)$ be $ w(D) \setminus \{X_a, X^\circ_a \colon \text{$X_a$ occurs in $D$}\}$. 


Now we claim that there is a formula $C$ such that $C$ does not contain literals of the form $\overline{X}_a$, atoms $ \{ X_a \colon a \in BH(\pi)\}$ can occur in $C $ only within the scope of modal operators, and
\begin{multline*}
\mathsf{GL} \vdash \bigwedge \lbrace \boxdot I_a \colon a \in H(\pi) \rbrace \wedge \bigwedge \lbrace \Box I_a \colon a \in BH(\pi) \rbrace \rightarrow  ((\overline{(\Gamma_1^\sharp)} \rightarrow C) \wedge (C \rightarrow \Gamma^\sharp_2)) \:,
\end{multline*}
\begin{multline}\label{eq:var}
\bigwedge\{ I^\prime_a \colon a \in  H(\pi) \cup BH(\pi)\} \Longrightarrow  w_X(C) \\ \cup 
\bigcup \{ w_a \colon a \in H(\pi) \} \cup \bigcup \{ w^\circ_a \colon a \in BH(\pi) \}  \subset w(\overline{\Gamma}_1) \cap w( \Gamma_2)\:.
\end{multline}
The proof is by induction on the structure of $\kappa$. 

Suppose the derivation $\kappa$ consists of a single sequent. 
If the single leaf has the form of an initial sequent:
\begin{gather*}
\AXC{ $(\top) \;\; \Gamma_1 \; | \; \top, \Gamma_2$}
\DisplayProof  \qquad
\AXC{ $(\bot) \;\; \Gamma_1 , \top  \; | \; \Gamma_2$}
\DisplayProof \\
\AXC{$(\bot) \;\; \Gamma_1, A, \overline{A} \; | \; \Gamma_2$}
\DisplayProof  \qquad
\AXC{$(\overline{A}) \;\;\Gamma_1, A \; | \;\overline{A} , \Gamma_2$}
\DisplayProof  \qquad
\AXC{$(\top) \;\;\Gamma_1 \; | \; A, \overline{A}, \Gamma_2 $}
\DisplayProof \:,
\end{gather*}
then the formula in parentheses is $C$.
If the single leaf $a$ is non-axiomatic, then $C$ is equal to $X_a$.

Otherwise, there exists the last application of an inference rule in $\kappa$. Consider the case that the conclusion of $\kappa$ is not connected by the back-link function. For instance, the last application of an inference rule has the form:
\begin{gather*}
\AXC{$\tau_1$}
\noLine
\UIC{\vdots}
\noLine
\UIC{$ \Gamma_1 \; | \; \Gamma_2 , A $}
\AXC{$\tau_2$}
\noLine
\UIC{\vdots}
\noLine
\UIC{$\Gamma_1 \; | \; \Gamma_2 , B $}
\LeftLabel{$\mathsf{\wedge}_r$}
\RightLabel{\:.}
\BIC{$\Gamma_1 \; | \; \Gamma_2 , A \wedge B$}
\DisplayProof 
\end{gather*}
By the induction hypothesis, there are formulas $C_1$ and $C_2$, which satisfy all the required conditions for $(\tau_1, d_{\tau_1})$ and $(\tau_2, d_{\tau_2})$, respectively.
We simply put $C \equiv C_1\wedge C_2$. The cases of other inference rules are completely analogous. The construction of $C$ for other rules is presented in parentheses in Fig. 2.

The only remaining case is that the conclusion of $\kappa$ is connected by the back-link function. By $b$ denote the non-axiomatic leaf connected by the back-link function with the root of $\kappa$. From the previous cases, there is a formula $C_0$ which satisfies the required conditions for the circular derivation $\pi_0$ obtained from $\pi$ by erasing the link connecting $b$ with the root of $\kappa$. Notice that $b \in BH(\pi_0)$. The formula $C_0$ does not contain literals of the form $\overline{X}_a$, and atoms $ \{ X_a \colon a \in BH(\pi)\} \cup \{ X_b\}$ can occur in $C_0 $ only within the scope of modal operators. In addition, we have the following:
\begin{multline*}
\mathsf{GL} \vdash \bigwedge \lbrace \boxdot I_a \colon a \in H(\pi) \rbrace \wedge \bigwedge \lbrace \Box I_a \colon a \in BH(\pi) \rbrace \wedge \Box I_b \\ \rightarrow ((\overline{(\Gamma_1^\sharp)} \rightarrow C_0) \wedge (C_0 \rightarrow \Gamma^\sharp_2)) \:,
\end{multline*}
\begin{multline}\label{eq:indvar}
\bigwedge\{ I^\prime_a \colon a \in  H(\pi) \cup BH(\pi)\}  \wedge I^\prime_b \Longrightarrow  w_X(C_0) \cup w_b^\circ \\ \cup 
\bigcup \{ w_a \colon a \in H(\pi) \} \cup \bigcup \{ w^\circ_a \colon a \in BH(\pi) \}  \subset w(\overline{\Gamma}_1) \cap w( \Gamma_2)\:.
\end{multline}
From Theorem \ref{Fixed-point}, there is a formula $F$ such that $w (F) \subset w^* (C_0) \setminus \{ X_b^\circ\}$ and
\[\mathsf{GL} \vdash \boxdot (X_b \leftrightarrow C_0) \leftrightarrow \boxdot (X_b \leftrightarrow F) \: .\]
Let us substitute $X_b$ in $C_0$ with $F$ and denote the result by $C$. 

We claim that $C$ satisfies the required conditions for $\pi$. Obviously, $C$ does not contain literals of the form $\overline{X}_a$, and atoms $ \{ X_a \colon a \in BH(\pi)\}$ occur in $C $ only within the scope of modal operators. 

Furthermore, notice that  $\mathsf{GL} \vdash I_b(F) \leftrightarrow ((\overline{(\Gamma_1^\sharp)} \rightarrow C) \wedge (C \rightarrow \Gamma^\sharp_2))$.
Applying admissibility of the L\"{o}b rule to the formula
\begin{multline*}
\mathsf{GL} \vdash \bigwedge \lbrace \boxdot I_a \colon a \in H(\pi) \rbrace \wedge \bigwedge \lbrace \Box I_a \colon a \in BH(\pi) \rbrace \wedge \Box I_b (F) \\\rightarrow  ((\overline{(\Gamma_1^\sharp)} \rightarrow C) \wedge (C \rightarrow \Gamma^\sharp_2)) \:,
\end{multline*}
we see that the assumption $\Box I_b(F)$ can be dropped out. 

It remains to check condition \eqref{eq:var}. Assume $\bigwedge\{ I^\prime_a \colon a \in  H(\pi) \cup BH(\pi)\}$. We claim that $w_X(C)  \subset w(\overline{\Gamma}_1) \cap w( \Gamma_2)$.
From \eqref{eq:indvar} we have 
\begin{gather*}
I_b^\prime \Rightarrow w_X(C_0) \cup w^\circ_b \subset w(\overline{\Gamma}_1) \cap w( \Gamma_2) \:.
\end{gather*} 
Substituting $w_b$ in $I_b^\prime$ with $\emptyset$, we get  $w_X(C_0) \subset w(\overline{\Gamma}_1) \cap w( \Gamma_2) $. Thus $I^\prime_b (w_X(C_0))$ is a valid statement.
The second substitution of $w_b$ in $I_b^\prime$ with $w_X(C_0)$, yields that $w_X(C_0) \cup w^\circ_X(C_0) \subset w(\overline{\Gamma}_1) \cap w( \Gamma_2) $. Consequently, 
$$w_X(C) \subset w_X(C_0)\cup w^\circ_X(F) \subset w_X(C_0) \cup w_X^\circ(C_0) \subset w(\overline{\Gamma}_1) \cap w( \Gamma_2)\:.$$
From \eqref{eq:indvar} we also have 
\begin{gather*}
I_b^\prime \Rightarrow \bigcup \{ w_a \colon a \in H(\pi) \} \cup \bigcup \{ w^\circ_a \colon a \in BH(\pi) \}  \subset w(\overline{\Gamma}_1) \cap w( \Gamma_2)\:.
\end{gather*}
Substituting $w_b$ in $I_b^\prime$ with $\emptyset$, we get 
\begin{gather*}
\bigcup \{ w_a \colon a \in H(\pi) \} \cup \bigcup \{ w^\circ_a \colon a \in BH(\pi) \}  \subset w(\overline{\Gamma}_1) \cap w( \Gamma_2)\:.
\end{gather*}
Hence condition \eqref{eq:var} is satisfied. 

Finally, if $\pi= (\kappa, d)$ is a split circular proof of $ \Gamma_1 \: | \: \Gamma_2 $, then the constructed formula $C$ is indeed an interpolant for this split sequent.
\end{proof}

\paragraph*{Acknowledgements.} I thank Lev Beklemishev and Maria Filatova for their 
personal support and helpful advises during the work on the paper. Thanks, Lev! A warm thank you, Maria!
\bibliographystyle{amsplain}
\bibliography{OnGL}

\end{document}